\documentclass[review]{elsarticle}
\usepackage[margin=1in]{geometry}

\usepackage{hyperref}
\usepackage{amsmath,amssymb,amsthm}
\usepackage{amsfonts}
\usepackage{mathrsfs}
\usepackage{comment}
\journal{Statistics \& Probability Letters}

\biboptions{authoryear}

\begin{document}

\newtheorem{theorem}{Theorem}[section]
\newtheorem{proposition}[theorem]{Proposition}
\newtheorem{lemma}[theorem]{Lemma}
\newtheorem{corollary}[theorem]{Corollary}
\newtheorem{assumption}[theorem]{Assumption}

\newcommand{\bbeta}{\boldsymbol{\beta}}
\newcommand{\aalpha}{\boldsymbol{\alpha}}
\newcommand{\ggamma}{\boldsymbol{\gamma}}
\newcommand{\argmax}{\mathop{\text{argmax}}}
\newcommand{\st}{\,|\,}

\frenchspacing

\begin{frontmatter}

\title{Asymptotic comparison of identifying constraints for Bradley-Terry models}

\author{Weichen Wu} \ead[mail]{wwu3@andrew.cmu.edu}
\author{Brian W. Junker}
\author{Nynke M.D. Niezink}
\address{Department of Statistics and Data Science, Carnegie Mellon University \\
Baker Hall,
Pittsburgh, PA 15213, USA}

%\author[mymainaddress]{Carnegie Mellon University}
%\address[mymainaddress]{5000 Forbes Avenue, Pittsburgh PA}

\begin{abstract}
The Bradley-Terry model is widely used for pairwise comparison data analysis. In this paper, we analyze the asymptotic behavior of the maximum likelihood estimator of the Bradley-Terry model in its logistic parameterization, under a general class of linear identifiability constraints. We show that the constraint requiring the Bradley-Terry scores for all compared objects to sum to zero minimizes the sum of the variances of the estimated scores, and recommend using this constraint in practice. 
\end{abstract}


\end{frontmatter}

%\linenumbers

\section{Introduction}

The Bradley-Terry model is widely used for analyzing pairwise comparison data \citep{bradley1952rank}. It assumes that all comparisons are independent, and endows every compared object with a score. The probability that one objects beats another in a comparison is a function of the scores of the two objects. The model is known in two forms. Originally, \cite{bradley1952rank} let the score of an object $i$ be $w_i^*>0$, and defined the probability that $i$ beats $j$ in a comparison by
\begin{equation}
    \mathbb{P}(\text{i beats j}) = \frac{w_i^*}{w_i^* + w_j^*}.
\label{bt-definition-1}
\end{equation}
By taking $\beta_i^* = \log (w_i^*)$, model \eqref{bt-definition-1} can be re-written as
\begin{equation}
\mathbb{P}(\text{i beats j}) = \sigma(\beta_i^*-\beta_j^*),
\label{bt-definition-2}
\end{equation}
in which $\sigma(\cdot)$ denotes the sigmoid function, $\sigma(t) = 1/(1+e^{-t})$. The $\beta$-pa\-ram\-etri\-zation \eqref{bt-definition-2} is used in the statistical software package \textsf{BradleyTerry2} in \textsf{R} \citep{turner2012bradley}, and empirical studies like \citep{liner2004methods} and \citep{varin2016statistical}. In this report, we focus on the $\beta$-parameterization, which has so far received less theoretical attention but more practical use. 

Probability \eqref{bt-definition-1} remains unchanged if all the $w$'s are multiplied by a constant and, equivalently, in \eqref{bt-definition-2}, if a constant is added to all the $\beta$'s. Therefore, in order to guarantee identifiability, the Bradley-Terry scores are usually estimated by constrained maximum likelihood estimation.
The reference constraint and the sum constraint are the two most widely-used constraints for Bradley-Terry models. For the reference constraint, one of the objects $i$ is set as a reference, and its score is set to $w_i = 1$ \citep{simons1999asymptotics} -- or, equivalently, $\beta_i = 0$. This constraint is used, for example, in the \textsf{R} packages %like 
\textsf{BradleyTerry2} \citep{turner2012bradley} and \textsf{prefmod} \citep{hatzinger2012prefmod}. For the sum constraint, the sum of all scores is set to a constant: in the $w$-parameterization, the sum of all $w$'s are usually set to $1$, as in theoretical papers \citep{bradley1952rank,ford1957solution,dykstra1960rank} and the \textsf{R} package \textsf{eba} \citep{wickelmaier2007eba}; in the $\beta$-parameterization, the sum of all $\beta$'s is usually set to $0$ \citep[e.g.,][]{varin2016statistical}. Clearly, the sum constraint under the $w$-parameterization is not equivalent to the sum constraint under the $\beta$-parameterization. 

The theoretical properties of the Bradley-Terry model have been studied extensively \citep{bradley1954incomplete,dykstra1960rank,ford1957solution,simons1999asymptotics,han2020asymptotic}. Most theoretical results are derived under the $w$-parameterization. For example, \cite{ford1957solution} gave the necessary and sufficient condition for the existence and uniqueness of the MLE under the sum constraint. \cite{dykstra1960rank} proved the consistency and asymptotic normality of the MLE under the sum constraint when the number of comparisons between all pairs of objects goes to infinity. \cite{simons1999asymptotics} proved the consistency and asymptotic normality under the reference constraint when the number of objects goes to infinity and all pairs of objects are compared the same number of times.  \cite{han2020asymptotic} generalized this result to the case where only a small proportion of pairs of objects are ever compared with each other.

In this paper, we extend these theoretical results to the $\beta$-parameterization \eqref{bt-mle} and address a practical issue. Specifically, we show in Section~\ref{s:the-mle} that the constrained MLE under the $\beta$-parameterization is unique under the condition posited by \cite{ford1957solution}. We then prove in  Section~\ref{s:asymptotics} the consistency and asymptotic normality of the estimators under a practical and flexible data collection method. Finally, in Section~\ref{s:variances} we discuss the estimated variance of the estimator under other linear constraints, and prove that the sum constraint is the unique constraint that minimizes the sum of the variances of the estimated scores for all objects. Moreover, as we show empirically in Section~\ref{s:example}, poor choice of a reference object for the reference constraint can artificially spread uncertainty to the other objects; this behavior is avoided when using the sum constraint, which tends to concentrate uncertainty of estimation on objects that have participated in fewer comparisons.  Thus, use of the sum constraint in the $\beta$ parameterization will help applied researchers in making inferences from pairwise comparison data. 

\section{Existence and uniqueness of the constrained MLE}\label{s:the-mle}

%In this section, we examine the conditions under which the constrained MLE exists and is unique.

Suppose there are $n$ objects being compared, and let $W_{ij}$  denote the number of times that $i$ beats $j$, and $V_{ij}$ the number of times that $i$ and $j$ are compared with each other, for $i,j \in \{1,\ldots, n\}$. The corresponding matrices, $\mathbf{W} \in \mathbb{N}_0^{n \times n}$ and $\mathbf{V} \in \mathbb{N}_0^{n \times n}$, satisfy $\mathbf{V} = \mathbf{W + W^\top}$.
The MLE for the $\beta$-parameterization under the linear constraint $\aalpha^\top \boldsymbol{\beta} = 0$, with $\aalpha$ a non-zero vector in  $\mathbb{R}^n$, is 
\begin{equation}
\hat{\boldsymbol{\beta}}(\mathbf{W}) = \mathop{\text{argmax}}_{\aalpha^\top \boldsymbol{\beta} = 0} \ell(\boldsymbol{\beta};\mathbf{W}),
\label{bt-mle}
\end{equation}
where $\ell(\boldsymbol{\beta};\mathbf{W})$ is the log-likelihood function
\begin{equation}
\ell(\boldsymbol{\beta};\mathbf{W}) = \sum_{i \neq j} W_{ij} \log \sigma(\beta_i - \beta_j).
\label{log-likelihood}
\end{equation} 
Note that by taking $\aalpha = \mathbf{1} = (1,1,...,1)^\top$, we obtain the sum constraint, and by taking $\alpha = \mathbf{e}_i = (0,0,...,1,0,...,0)^\top$, we obtain the reference constraint.

The existence and uniqueness of the constrained MLE requires sufficient knowledge of the comparisons among the objects. The following assumption was proposed by \cite{ford1957solution} as a necessary and sufficient condition for the existence and uniqueness of the constrained MLE under the $w$-parameterization.

\begin{assumption}\label{assumption-A} In every possible partition of the objects into two nonempty subsets, some object in the second set has been preferred at least once to some object in the first.
\end{assumption}

This assumption can also be interpreted in graph-theoretical terms. Define the (directed) comparison graph $\mathcal{G} = (V,E)$, where the vertex set $V = \{1,2,...,n\}$ represents the objects, and the edge set $E$ contains all edges $i \to j$ such that $i$ has been preferred at least once over $j$. Then Assumption \ref{assumption-A} essentially states that $\mathcal{G}$ is strongly connected. That is, for every pair of objects $(i,j)$, there exists a directed path from $i$ to $j$. 
\cite{ford1957solution} proved the following theorem.

\begin{theorem} Assumption \ref{assumption-A} is necessary and sufficient for the existence and uniqueness of the constrained MLE for the Bradley-Terry model under the $w$-parameterization and the sum constraint $\mathbf{1}^\top \mathbf{w} = 1$.

\label{theorem-1}
\end{theorem}

Theorem \ref{theorem-1} guarantees the existence and uniqueness of the constrained MLE under reasonable conditions. The violation of Assumption \ref{assumption-A} means either that a group of objects has never been compared with the rest, or that this group has never been preferred over the rest. In both cases, it is hard to evaluate the comparative strength between the two groups, and it would be more reasonable to analyze them separately. A similar conclusion follows for the $\beta$-parameterization.

\begin{corollary} Assumption \ref{assumption-A} is necessary and sufficient for the existence and uniqueness of the constrained MLE for the Bradley-Terry model under the $\beta$-parameterization and the constraint  $\aalpha^\top \boldsymbol{\beta} = 0$  %\eqref{bt-mle} 
for any $\boldsymbol{\alpha}$ such that $\mathbf{1}^\top \boldsymbol{\alpha} \neq 0$.

\label{corollary-1}
\end{corollary}

The proof of Corollary \ref{corollary-1} is given in \ref{appendix-A}.

\section{Asymptotic properties of the constrained MLE}\label{s:asymptotics}

The asymptotic behavior of the constrained MLE under the $w$-parameterization has been studied under various asymptotic regimes, implied by different data collection designs. %\cite{bradley1954incomplete} pointed out that for a fixed number of objects, when every pair of objects are compared for the same number of times, the estimator under the sum constraint is consistent and asymptotically normal as the number of total comparisons goes to infinity. \cite{dykstra1960rank} generalized this result to the case where different pairs of objects are compared for different number of times. Furthermore, \cite{simons1999asymptotics} proved that the estimator under the reference constraint is consistent and asymptotically normal when the number of objects go to infinity and all pairs of objects are compared for the same number of times; \cite{han2020asymptotic} generalized this result to the case where the comparison graph $\mathcal{G}$ can be sparse.
The data collection design we will consider is as follows: suppose that  $n$ objects are compared, with $n$ fixed. Also, suppose that there are $S$ subjects making the comparisons, and that subject $s$ makes $V^{(s)}$ random comparisons, where the $V^{(1)},\dots,V^{(S)}$ are independent and identically distributed, with $0<\mathbb{E}[V^{(s)}] <\infty$ and $\text{Var}[V^{(s)}] < \infty$ (the i.i.d. assumption can be weakened, but we will maintain it in this paper for simplicity). %After this,  $V^{(s)}$ pairs of objects are presented to the subject for comparison. 
Let $\mathbf{V}^{(s)}$ denote the comparison matrix corresponding to subject $s$, with its elements $V^{(s)}_{ij} = V^{(s)}_{ji}$ representing the number of times subject $s$ compares the pair of objects $(i,j)$. 

We here primarily consider the asymptotic properties of the MLE of $\bbeta$ under the sum constraint $\mathbf{1}^\top \bbeta = 0$, as $S\rightarrow\infty$. We focus on this constraint because, as will be shown in the next section, it has the unique advantage of minimizing the sum of the variances of the estimated scores. We will discuss how the results can be generalized to other linear constraints at the end of this section. We will discuss these properties when the total number of comparisons goes to infinity. 
We make the following assumption.

\begin{assumption} \label{equality-assumption}
There exists a matrix $\mathbf{P} \in [0,1]^{n \times n}$, such that for all subjects $s$,
\begin{equation}
    \mathbb{E}[\mathbf{V}^{(s)} \st V^{(s)}] = V^{(s)} \mathbf{P}.
\end{equation}
\end{assumption}
In matrix $\mathbf{P}$, element $P_{ij}$, represents the overall frequency that the pair $(i,j)$ is compared in the survey. The equality assumption basically states that this frequency holds for every subject in the survey. This is true when all subjects are asked to compare every pair of objects, or i.i.d. pairs of objects are chosen for comparisons by the subjects.

The following proposition shows that as the number of subjects $S$ grows to infinity, the overall comparison matrix $\mathbf{V}$ reflects the relative frequency with which each pair of objects is compared.

\begin{proposition}
Let $V = \sum_{s=1}^S V^{(s)}$ denote the total number of comparisons made by all subjects. Under the aforementioned data collection method, as the number of subjects $S$ grows to infinity,
\begin{equation}
\frac{\mathbf{V}}{V} \xrightarrow{p} \mathbf{P}.
\end{equation}
\end{proposition}

\emph{Proof.} 
According to the data collection design, $V^{(1)},\ldots,V^{(S)}$ are i.i.d.\ random variables with finite variance. According to Assumption \ref{assumption-A}, for every pair of objects $(i,j)$ and every subject $s \in \{1,2,...S\}$, $V_{ij}^{(s)}$ are also i.i.d.\ random variables with finite variance. Therefore, by the weak law of large numbers and the continuous mapping theorem, we have that
\begin{equation}
\pushQED{\qed} 
\frac{V_{ij}}{V} = \frac{\sum_{s=1}^S V_{ij}^{(s)}}{\sum_{s=1}^S V^{(s)}}
= \frac{\sum_{s=1}^S V_{ij}^{(s)}}{S} \times \frac{S}{\sum_{s=1}^S V^{(s)}}
\xrightarrow{p} P_{ij}\mathbb{E}[V^{(s)}] \times \frac{1}{\mathbb{E}[V^{(s)}]} = P_{ij}. \qedhere \popQED
\end{equation}

We now prove the following theorem, which shows the consistency of the constrained MLE.

\begin{theorem}\label{theorem-2}
Under the aforementioned data collection method, as long as the constrained MLE \eqref{bt-mle} exists and is unique, as $S \to \infty$, it satisfies
\begin{equation}
    \hat{\boldsymbol{\beta}} \xrightarrow{p} \boldsymbol{\beta}^*.
\end{equation}

\end{theorem}

\begin{proof}
Let $\mathbf{W}^{(s)}$ denote the contingency matrix corresponding to subject $s$, such that $W_{ij}^{(s)}$ represents the number of times that subject $s$ chooses object $i$ over object $j$. Define $\mathbf{W}^* \in \mathbb{R}^{n \times n}$  as its expectation, 
\begin{equation}
W_{ij}^* =\mathbb{E}[W_{ij}^{(s)}] = \mathbb{E}[\mathbb{E}[W_{ij}^{(s)} \st V_{ij}^{(s)}]] = \mathbb{E}[V_{ij}^{(s)} \sigma(\beta_i^* - \beta_j^*)],
\end{equation}
and let 
$\overline{\mathbf{W}} = \frac{1}{S} \mathbf{W} = \frac{1}{S}\sum_{s=1}^S \mathbf{W}^{(s)},
$
denote the sample mean of the contingency matrix of all subjects. Then,  since the variance of $W_{ij}^{(s)}$ is finite, according to the weak law of large numbers,
\begin{equation}
\overline{\mathbf{W}} \xrightarrow{p} \mathbf{W}^*.
\end{equation}
Since the log-likelihood function $\ell(\bbeta;\overline{\mathbf{W}})$ is continuous in both $\bbeta$ and $\overline{\mathbf{W}}$, $\hat{\boldsymbol{\beta}}(\overline{\mathbf{W}})$ is a continuous function of $\overline{\mathbf{W}}$. Hence, according to the continuous mapping theory, we have that
\begin{equation}
\hat{\boldsymbol{\beta}}(\mathbf{W}) = \hat{\boldsymbol{\beta}}(\overline{\mathbf{W}}) \xrightarrow{p} \hat{\boldsymbol{\beta}}(\mathbf{W}^*),
\end{equation}
where the first equality follows from the definition of the estimator. It remains to show that $\hat{\boldsymbol{\beta}}(\mathbf{W}^*) = \boldsymbol{\beta}^*$. By definition, the likelihood $\ell(\boldsymbol{\beta},\mathbf{W}^*)$ is given by
\begin{equation}
\begin{split}
\ell(\boldsymbol{\beta},\mathbf{W}^*) &= \sum_{1 \leq i<j \leq n} \left[W_{ij}^* \log \sigma(\beta_i - \beta_j) + W_{ji}^* \log \sigma(\beta_j - \beta_i) \right]\\
&= \sum_{1 \leq i<j \leq n} P_{ij}\, \mathbb{E}[V^{(s)}]  \left[\sigma(\beta_i^* - \beta_j^*) \log \sigma(\beta_i - \beta_j) + \sigma(\beta_j^* - \beta_i^*) \log \sigma(\beta_j - \beta_i)\right].
\end{split}
\end{equation}
Using the fact that $p \log q + (1-p) \log(1-q)$ is maximized when $p=q$, we have that
\begin{equation}
\begin{split}
\ell(\boldsymbol{\beta},\mathbf{W}^*) & \leq \sum_{1 \leq i<j \leq n} P_{ij}\, \mathbb{E}[V^{(s)}] \left[\sigma(\beta_i^* - \beta_j^*) \log \sigma(\beta_i^* - \beta_j^*) + \sigma(\beta_j^* - \beta_i^*) \log \sigma(\beta_j^* - \beta_i^*)\right]\\
& = \ell(\boldsymbol{\beta}^*,\mathbf{W}^*),
\end{split}
\end{equation}
which effectively means that $\hat{\boldsymbol{\beta}}(\mathbf{W}^*) = \boldsymbol{\beta}^*$, since $\hat{\boldsymbol{\beta}}(\mathbf{W}^*)$ is unique.
\end{proof}

Next, we construct the asymptotic distribution of the MLE under the sum constraint. Let $H(\bbeta,\mathbf{W})$ denote the Hessian matrix of the log-likelihood function,
\begin{equation}
H(\bbeta,\mathbf{W}) = \frac{\partial^2 \ell(\bbeta,\mathbf{W})}{\partial \bbeta^{2}}
= -\sum_{1\leq i < j \leq n} V_{ij} \sigma(\beta_i-\beta_j)\sigma(\beta_j - \beta_i)(\mathbf{e}_i - \mathbf{e}_j )(\mathbf{e}_i - \mathbf{e}_j )^\top.
\label{eq:hessian-matrix}
\end{equation}
For simplicity, we will use $H(\bbeta)$ to represent $H(\bbeta;\mathbf{W})$ when there is no ambiguity. The following lemma states two properties of $H(\bbeta;\mathbf{W})$ and serve as the basis of our analysis of the asymptotic properties of the constrained MLE. The proof of this lemma is given in \ref{appendix-C}.

\begin{comment}
\begin{lemma}\label{lemma-1}
If Assumption \ref{assumption-A} holds, then for any $\bbeta \in \mathbb{R}^n$, the Hessian $H(\bbeta)$ is negative semi-definite, and
%
\[
\ker(H(\bbeta)) = \textnormal{span}(\mathbf{1}).
\]
\end{lemma}
\end{comment}

\begin{lemma} \label{lemma-2} 
Let $H^\dagger(\bbeta)$ denote the pseudoinverse of $H(\bbeta)$.\footnote{A matrix $\mathbf{X} \in \mathbb{R}^{n \times n}$ is called the pseudoinverse of a matrix $\mathbf{A} \in \mathbb{R}^{n \times n}$, denoted as $\mathbf{X = A^\dagger} $, if the following four conditions hold: 1) $\mathbf{AXA = A}$,
2) $\mathbf{XAX = X}$,
3) $\mathbf{(AX)^\top = AX}$, and
4) $\mathbf{(XA)^\top = XA}$. \citep[,p.27]{ben2003generalized}.}
%
%For the definition and properties of inverse, please refer to \ref{appendix-A}} 
If Assumption \ref{assumption-A} holds, then for any scores $\tilde{\boldsymbol{\beta}} \in \mathbb{R}^n$ and $\boldsymbol{\beta} \in \mathbb{R}^{n}$ such that $\mathbf{1}^\top \boldsymbol{\beta} = 0$, it holds that $H^\dagger(\tilde{\boldsymbol{\beta}})H(\tilde{\boldsymbol{\beta}}) \boldsymbol{\beta} = \boldsymbol{\beta}$. Furthermore, $H^\dagger(\tilde{\boldsymbol{\beta}}) \mathbf{1} = 0$. 
\end{lemma}

The following theorem gives the asymptotic distribution of the MLE of $\bbeta$ under the sum constraint.

\begin{theorem} \label{theorem-3}
Let $\hat{\bbeta}$ denote the MLE of $\bbeta$ under the sum constraint $\mathbf{1}^\top \boldsymbol{\beta} = 0$. Then as long as $\hat{\bbeta}$ exists and is unique, as $S \to \infty$, and therefore $V \to \infty$, we have that
\begin{equation}
\sqrt{V}(\hat{\boldsymbol{\beta}} - \boldsymbol{\beta}^*) \xrightarrow{d} \mathcal{N}(0,\mathcal{I}^\dagger(\boldsymbol{\beta}^*)).
\end{equation}
\end{theorem}
\begin{proof} 
By definition, the MLE under the sum constraint is obtained by maximizing the Lagrangian function
\begin{equation}
\mathscr{L}(\boldsymbol{\beta},\mu;\mathbf{W}) = \ell(\boldsymbol{\beta},\mathbf{W}) + \mu \mathbf{1}^\top \boldsymbol{\beta}.
\end{equation}
By first-order condition and Lemma \ref{lemma-2}, we find that
\begin{equation}
\mathbf{0} =H^\dagger \nabla_{\boldsymbol{\beta}} \mathscr{L}(\boldsymbol{\beta},\mu;\mathbf{W}) \Big |_{\boldsymbol{\beta}}= H^\dagger(\tilde{\boldsymbol{\beta}}) (\nabla \ell(\hat{\boldsymbol{\beta}},\mathbf{W}) + \mu \mathbf{1}) = H^\dagger(\tilde{\boldsymbol{\beta}})\nabla \ell(\hat{\boldsymbol{\beta}},\mathbf{W}).
\label{eq-1}
\end{equation}
According to Lagrange's mean-value theorem, there exists a $\tilde{\boldsymbol{\beta}}$ between $ \hat{\boldsymbol{\beta}}$ and $\boldsymbol{\beta}^*$ such that

\begin{equation}
\begin{split}
H^\dagger(\tilde{\boldsymbol{\beta}})\nabla \ell(\boldsymbol{\beta}^*;\mathbf{W})& = H^\dagger(\tilde{\boldsymbol{\beta}})\nabla \ell(\hat{\boldsymbol{\beta}},\mathbf{W}) + H^\dagger(\tilde{\boldsymbol{\beta}})H(\tilde{\boldsymbol{\beta}})(\boldsymbol{\beta}^*  - \hat{\boldsymbol{\beta}}) \\
&= H^\dagger(\tilde{\boldsymbol{\beta}})\nabla \ell(\hat{\boldsymbol{\beta}},\mathbf{W}) +  \boldsymbol{\beta}^*  - \hat{\boldsymbol{\beta}} = \boldsymbol{\beta}^*  - \hat{\boldsymbol{\beta}},
\end{split}
\label{eq-2}
\end{equation}
according to \eqref{eq-1} and Lemma \ref{lemma-2}.
Meanwhile, since $H(\bbeta;\mathbf{W})$ is a linear functions of $\mathbf{W}$, we have $H(\bbeta;\overline{\mathbf{W}}) =  H(\bbeta;\mathbf{W}) / S$, and therefore $H^\dagger(\bbeta;\overline{\mathbf{W}}) = S \cdot H^\dagger(\bbeta;\mathbf{W})$. As a consequence,
\begin{equation}\label{eq-3}
\sqrt{V}(\hat{\boldsymbol{\beta}} - \boldsymbol{\beta}^*) = \sqrt{\frac{V}{S}}\left[S\cdot H^\dagger(\tilde{\boldsymbol{\beta}};\mathbf{W})\right] \left[\frac{1}{\sqrt{S}} \nabla \ell(\boldsymbol{\beta}^*;\mathbf{W}) \right] = \sqrt{\frac{V}{S}}\left[H^\dagger(\tilde{\boldsymbol{\beta}};\overline{\mathbf{W}})\right] \left[\frac{1}{\sqrt{S}} \nabla \ell(\boldsymbol{\beta}^*;\mathbf{W}) \right].
\end{equation}
Now that we have linked the difference between the true score $\bbeta^*$ and the estimator $\hat{\bbeta}$ with the gradient of the log-likelihood, we will use the law of large numbers and the central limit theorem to construct the asymptotic distribution of the MLE. We can write the gradient of the log-likelihood function $\nabla \ell(\boldsymbol{\beta}^*;\mathbf{W})$ as the sum of the gradients of the subject-specific log-likelihoods,
\begin{equation}
\nabla \ell(\boldsymbol{\beta}^*;\mathbf{W}) = \sum_{s=1}^S \nabla \ell(\boldsymbol{\beta}^*;\mathbf{W}^{(s)}). 
\end{equation}
It is easy to verify that $\mathbb{E}\left[\nabla \ell(\boldsymbol{\beta}^*;\mathbf{W}^{(s)}) \right] = 0$ and $\mathbf{Var}\left[\nabla \ell(\boldsymbol{\beta}^*;\mathbf{W}^{(s)}) \right]  = \mathbb{E}[V^{(s)}]\, \mathcal{I}(\boldsymbol{\beta}^*)$, where $\mathcal{I}(\bbeta^*)$ denotes
%
%\[
%\mathcal{I}(\bbeta^*) =-\frac{\partial^2 \ell(\bbeta^*,\mathbf{P})}{\partial \bbeta^{*2}}= \sum_{i \neq j} P_{ij} \sigma(\beta_i^* - \beta_j^*) \sigma(\beta_j^* - \beta_i^*)(\mathbf{e}_i - \mathbf{e}_j )(\mathbf{e}_i - \mathbf{e}_j )^\top.
%\]
the Fisher information matrix evaluated in $\bbeta^*$. Therefore, it follows from the central limit theorem that
\begin{equation}\label{eq-4}
\frac{1}{\sqrt{S}} \nabla \ell(\boldsymbol{\beta}^*;\mathbf{W})  \xrightarrow{d} \mathcal{N}(0,\mathbb{E}[V^{(s)}]\mathcal{I}(\boldsymbol{\beta}^*)).
\end{equation}
On the other hand, since $\hat{\boldsymbol{\beta}} \xrightarrow{p} \boldsymbol{\beta}^*$ and $\overline{\mathbf{W}} \xrightarrow{p} \mathbf{W}^*$, by the continuous mapping theorem we find that
\begin{equation}\label{eq-5}
H^\dagger(\tilde{\boldsymbol{\beta}};\overline{\mathbf{W}}) \xrightarrow{p} H^\dagger(\boldsymbol{\beta}^*;\mathbf{W}^*) = -\frac{1}{\mathbb{E}[V^{(s)}]}\mathcal{I}^\dagger(\boldsymbol{\beta}^*).
\end{equation}
Finally, since $\frac{V}{S} \xrightarrow{p} \mathbb{E}[V^{s}]$, by combining the results from \eqref{eq-3},\eqref{eq-4} and \eqref{eq-5}, we obtain that 
\begin{equation}
\sqrt{V}(\hat{\boldsymbol{\beta}} - \boldsymbol{\beta}^*) \xrightarrow{d} \mathcal{N}(0, \mathcal{I}^\dagger(\boldsymbol{\beta}^*)\mathcal{I}(\boldsymbol{\beta}^*)\mathcal{I}^\dagger(\boldsymbol{\beta}^*)) = \mathcal{N}(0,\mathcal{I}^\dagger(\boldsymbol{\beta}^*)).
\end{equation}
\end{proof}

Theorem \ref{theorem-3} is important because it implies that, under a flexible data collection method, we can evaluate the uncertainty of $\hat{\boldsymbol{\beta}}$ as follows. Since
\[
V \cdot \mathbf{Var}[\hat{\boldsymbol{\beta}}] = \mathbf{Var}\left[\sqrt{V}(\hat{\boldsymbol{\beta}} - \boldsymbol{\beta}^*)\right]\approx \mathcal{I}^\dagger(\boldsymbol{\beta}^*)
\]
and $H(\hat{\boldsymbol{\beta}},\overline{\mathbf{W}}) \xrightarrow{p} \mathbb{E}[V^{(s)}]\,\mathcal{I}(\boldsymbol{\beta}^*)$, we find that the asymptotic variance of the MLE for $\bbeta$ under the sum constraint $\mathbf{1}^\top \boldsymbol{\beta} = 0$ is the pseudoinverse of the Fisher information matrix, and
\begin{equation}
\mathbf{Var}[\hat{\boldsymbol{\beta}}] \approx -H^\dagger(\hat{\boldsymbol{\beta}};\mathbf{W}).
\label{asym-variance}
\end{equation}
This is a generalization of the case for identifiable statistical models, where the asymptotic variance of the MLE is the (standard) inverse of the Fisher information matrix. This result is specific to the sum constraint and the $\beta$-parameterization: we will see that, for other kinds of linear constraints under the $\beta$-parameterization, the asymptotic variance of the estimator is the reflexive generalized inverse\footnote{A matrix $\mathbf{X} \in \mathbb{R}^{n \times n}$ is called the reflexive generalized inverse of a matrix $\mathbf{A} \in \mathbb{R}^{n \times n}$, if the following two conditions hold: 1) $\mathbf{AXA = A}$, and
2) $\mathbf{XAX = X}$.  \citep[,p.27]{ben2003generalized}.} %\footnote{The definition of reflexive generalized inverse matrix is shown in the appendix; the details of this claim is shown in the next section.}
of the Fisher information matrix (see Theorem \ref{eq:generalconstraint}). %Note that under the $w$-parameterization, the asymptotic variance of the estimator is also the reflexive generalized inverse of the Fisher information matrix $\mathcal{I}(\mathbf{w}^*)$.\footnote{\cite{bradley1954incomplete} and \cite{dykstra1960rank} gave expressions for the asymptotic variance of $\hat{\mathbf{w}}$, but did not point out explicitly that it is the reflexive generalized inverse of the Fisher information matrix. It is easy, though, to verify this result.}

\section{Variance for the Bradley-Terry scores under different linear constraints}\label{s:variances}
In the previous section, we justified the use of the pseudoinverse of $-H(\hat{\bbeta})$ as the estimated variance of the Bradley-Terry scores under the sum constraint $\mathbf{1}^\top \bbeta = 0$. In this section, we will discuss the estimated variance under the more general class of linear constraints, and show that under the $\beta$-parameterization, the sum constraint minimizes the sum of the variances of the estimated scores for all objects. 

Suppose that the Bradley-Terry scores are subject to the constraint $\aalpha^\top \bbeta = 0$, in which $\aalpha \in \mathbb{R}^n $ is neither a multiple of $\mathbf{1}$ nor perpendicular to $\mathbf{1}$. If $\aalpha = c \mathbf{1}$ for some constant $c \neq 0$, the constraint is equivalent to the sum constraint. If $\mathbf{1}^\top \aalpha =0$, a constant can still be added to the scores of all subjects, with the constraint remaining satisfied and the log-likelihood function unchanged; therefore, the MLE would still not be unique. Use $\hat{\ggamma}$ to denote the MLE under this constraint. We will firstly prove the following theorem, which characterizes the estimated variance under this constraint.

\begin{theorem} \label{eq:generalconstraint}
The estimated variance $\widehat{\mathbf{Var}}(\hat{\ggamma})$ of  $\hat\ggamma$, the MLE (for $\beta$) under the constraint $\aalpha^\top \bbeta = 0$, where $\aalpha \notin \textnormal{span}(\mathbf{1})$ and $\mathbf{1}^\top \aalpha \neq 0$, is a reflexive generalized inverse of $-H(\hat{\ggamma})$. %Furthermore, the Hessian matrices for $\hat{\ggamma}$ and $\hat{\bbeta}$ are the same: $H(\hat{\ggamma}) = H(\hat{\bbeta})$.
\end{theorem}

\begin{proof}
For any $\hat{\bbeta},\hat{\ggamma}$ such that $\mathbf{1}^\top \hat{\bbeta} = 0 = \aalpha^\top \hat{\ggamma}$, we can construct the following bijection:
\begin{equation}
\begin{split}
&\hat{\bbeta} \mapsto \hat{\ggamma} = \hat{\bbeta} - \frac{\aalpha^\top \hat{\bbeta}}{\mathbf{1}^\top\aalpha } \mathbf{1} = (\mathbf{I} - \frac{\mathbf{1}\aalpha^\top}{\mathbf{1}^\top \aalpha})\hat{\bbeta}\, ,\\
&\hat{\ggamma} \mapsto \hat{\bbeta} = \hat{\ggamma} - \frac{\mathbf{1}^\top \hat{\ggamma}}{n} \mathbf{1}.
\end{split}
\label{eq:bijection}
\end{equation}
Then it is easy to verify that $\ell(\hat{\bbeta};\mathbf{W}) = \ell(\hat{\ggamma};\mathbf{W})$. Therefore, if $\hat{\bbeta}$ is the MLE under the sum constraint $\mathbf{1}^\top \hat{\bbeta} = 0$, then $\hat{\ggamma}$ is the MLE under the constraint $\aalpha^\top \bbeta = 0$. Meanwhile, it is easy to verify from \ref{eq:bijection} that the difference between scores of every pair of objects does not change under different constraints. Hence, the Hessian matrix (and therefore, the Fisher information matrix) \ref{eq:hessian-matrix} remains the same when we evaluate at $\hat{\ggamma}$ instead of $\hat{\bbeta}$. Thus, $H^\dagger(\bbeta) = H^\dagger(\ggamma)$. 

According to the transformation of variance,
\begin{equation}
\begin{split}
    \widehat{\mathbf{Var}}(\hat{\ggamma})&= (\mathbf{I} - \frac{\mathbf{1}\aalpha^\top}{\mathbf{1}^\top \aalpha}) \widehat{\mathbf{Var}}(\hat{\bbeta})(\mathbf{I} - \frac{\mathbf{1}\aalpha^\top}{\mathbf{1}^\top \aalpha})^\top \\
    &= -(\mathbf{I} - \frac{\mathbf{1}\aalpha^\top}{\mathbf{1}^\top \aalpha}) H^\dagger(\hat{\ggamma}) (\mathbf{I} - \frac{\aalpha\mathbf{1}^\top}{\mathbf{1}^\top \aalpha}).
\end{split}
\label{alpha-var}
\end{equation}
It follows from Lemma \ref{lemma-2} that $H^\dagger(\hat{\ggamma}) \mathbf{1} = 0$ and, therefore, $\mathbf{1}^\top H^\dagger(\hat{\ggamma}) = 0$.
Since the $H^\dagger(\hat{\ggamma})$ has the same column space and null space as $H(\hat{\ggamma})$ (page 36 of \citep{ben2003generalized}), we also have that $H(\hat{\ggamma}) \mathbf{1} = 0$ and $\mathbf{1}^\top H(\hat{\ggamma}) = 0$. It therefore can be verified that
\begin{equation}
\widehat{\mathbf{Var}}(\hat{\ggamma})(-H(\hat{\bbeta}))\widehat{\mathbf{Var}}(\hat{\ggamma}) = \widehat{\mathbf{Var}}(\hat{\ggamma})
\label{ref-inv-cond-1}
\end{equation}
and that
\[
H(\hat{\ggamma})\widehat{\mathbf{Var}}(\hat{\ggamma})H(\hat{\ggamma}) = -H(\hat{\ggamma}).
\]
Consequently, the estimated variance $\widehat{\mathbf{Var}}(\hat{\ggamma})$ is a reflexive generalized inverse of $-H(\hat{\ggamma})$. 
\end{proof}
 
\iffalse
The reflexive generalized inverse is not unique. For the constraint $\aalpha ^\top \hat{\ggamma} = 0$, $\widehat{\mathbf{Var}}(\hat{\ggamma})$ is the one that satisfies
%
\begin{equation}
\widehat{\mathbf{Cov}}(\hat{\ggamma},\aalpha^\top \hat{\ggamma}) = \widehat{\mathbf{Var}}(\hat{\ggamma})\aalpha = 0.
\label{specify-ref-gen-inv}
\end{equation}
Hence, for different $\aalpha$, equation \eqref{specify-ref-gen-inv} implies that the estimated variances are different. Meanwhile, the pseudoinverse of a matrix is unique, so only if $\aalpha$ is a multiple of $\mathbf{1}$, the estimated variance is the pseudoinverse of $-H(\hat{\ggamma})$. 
%
\fi
From equation \eqref{alpha-var}, we see that $\widehat{\mathbf{Var}}(\hat{\ggamma})$ varies with $\aalpha$. Equation \eqref{alpha-var} also indicates the relationship between $\widehat{\mathbf{Var}}(\hat{\ggamma})$ and $\widehat{\mathbf{Var}}(\hat{\bbeta})$, and implies the following theorem.

\begin{theorem}\label{theorem-4}
The sum constraint $\mathbf{1}^\top \bbeta = 0$ is the one that minimizes the sum of the variances of all scores. 
\end{theorem}
\emph{Proof}.  Consider
\begin{equation}
\begin{split}
\sum_{i=1}^n \widehat{\text{Var}}(\hat{\gamma}_i) &= \text{Tr}(\widehat{\mathbf{Var}}(\hat{\ggamma}))\\
&= \text{Tr}(\widehat{\mathbf{Var}}(\hat{\bbeta})) - \text{Tr}(\frac{\mathbf{1}\aalpha^\top}{\mathbf{1}^\top \aalpha}\widehat{\mathbf{Var}}(\hat{\bbeta})) - \text{Tr}(\widehat{\mathbf{Var}}(\hat{\bbeta})\frac{\aalpha\mathbf{1}^\top}{\mathbf{1}^\top \aalpha}) +\text{Tr}(\frac{\mathbf{1}\aalpha^\top}{\mathbf{1}^\top \aalpha}\widehat{\mathbf{Var}}(\hat{\bbeta})\frac{\aalpha\mathbf{1}^\top}{\mathbf{1}^\top \aalpha})\\
&= \sum_{i=1}^n \widehat{\text{Var}}(\hat{\beta}_i) - \frac{1}{1^\top \aalpha} \text{Tr}(\aalpha^\top \widehat{\mathbf{Var}}(\hat{\bbeta})\mathbf{1})- \frac{1}{1^\top \aalpha} \text{Tr}(\mathbf{1}^\top \widehat{\mathbf{Var}}(\hat{\bbeta})\aalpha)\\
&\quad + \frac{1}{(\mathbf{1}^\top \aalpha)^2} \text{Tr}(\aalpha^\top \widehat{\mathbf{Var}}(\hat{\bbeta}) \aalpha).
\end{split}
\end{equation}

Therefore, according to Lemma \ref{lemma-2}, as long as $\mathbf{1}^\top \aalpha \neq 0$ and $\aalpha \notin \text{span}(\mathbf{1})$,
\begin{equation*}
\pushQED{\qed} 
\sum_{i=1}^n \widehat{\text{Var}}(\hat{\gamma}_i)= \sum_{i=1}^n \widehat{\text{Var}}(\hat{\beta}_i) - 0 - 0 + \frac{\aalpha^\top \widehat{\mathbf{Var}}(\hat{\bbeta}) \aalpha}{(\mathbf{1}^\top \aalpha)^2}
> \sum_{i=1}^n \widehat{\text{Var}}(\hat{\beta}_i). \qedhere \popQED
\end{equation*}

Theorem \ref{theorem-4} strongly suggests that we use the sum constraint in practice. Currently, some statistical packages, like \textsf{BradleyTerry2} in \textsf{R}, use the reference constraint ($\beta_1$ = 0). We illustrate the difference between these two practices by analyzing the data collected in a pairwise comparison survey.

\section{Example: pairwise comparison of activities}\label{s:example}

In March 2020, a group of faculty members at the Department of Statistics and Data Science of Carnegie Mellon University surveyed the department members about their favorite online activities during the COVID-19 pandemic. They used a so-called wiki survey \citep{salganik2012wiki}, in which the respondents (subjects) could vote between pairs of activities (objects) that were presented to them randomly, or enter their own activities for comparison. The survey lasted for about one month, and 52 respondents voted on 21 activities.  

We fit the Bradley-Terry model to these data under both the reference constraint and the sum constraint. For the reference constraint, the activity ``Virtual Meeting" is set as the reference by default of the \textsf{BradleyTerry2} package in \textsf{R}, because it happens to be the first activity (object) in the data set. We then estimate the variances of the scores, and plot the 2-standard error confidence intervals of the estimators under both constraints. The results are show in Figure \ref{fig:rank}.

The graphs show that under the reference constraint (the plot on the left), since the activity ``Virtual Meeting" is set as the reference, its score is $0$ without any uncertainty. However, the scores of all other activities have very wide confidence intervals, suggesting that we are very uncertain about their estimated scores. Thus it is difficult to say that one object dominates another in preference order.

On the other hand, under the sum constraint (the plot on the right), the confidence interval for the score of ``Virtual Meeting" is quite wide, while those of all other activities are relatively narrow. This shows that we are not certain about the estimate of the score of ``Virtual Meeting", but quite sure of others, among which ``Online origami" is the one with the most uncertainty.
The reason for this phenomenon is that ``Virtual Meeting" happens to be the activity that was involved in the least number of comparisons. The respondents made 750 comparisons in total, but ``Virtual Meeting" was only compared for 8 times. The greater uncertainty associated with the small number of comparisons becomes obvious under the sum constraint. However, under the reference constraint, when ``Virtual Meeting" is set as the reference activity, this uncertainty is artificially propagated to all of the other activities.

\begin{figure}
\begin{minipage}{0.5\textwidth}
    \centering
    \includegraphics[width = \textwidth]{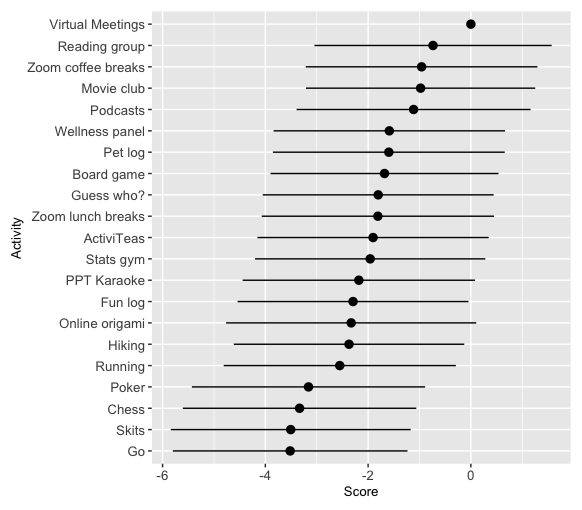}
\end{minipage}
\begin{minipage}{0.5\textwidth}
\centering
    \includegraphics[width = \textwidth]{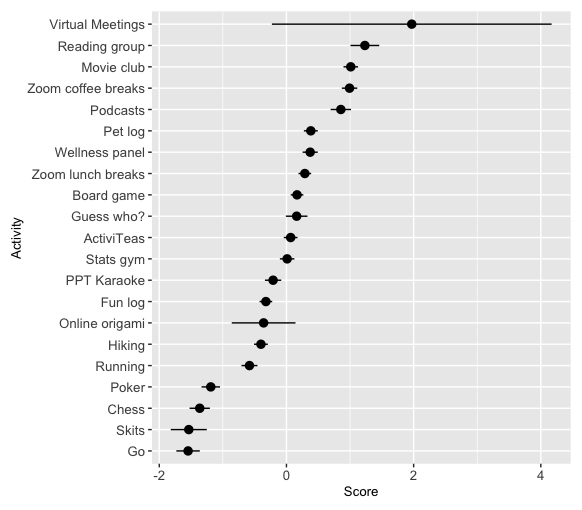}
\end{minipage}
\caption{The 2-standard-error confidence intervals for the estimated scores under the reference constraint (left) and the sum constraint (right).}
\label{fig:rank}
\end{figure}

\section{Discussion}

In this paper, we studied the asymptotic properties of the MLE for the Bradley-Terry model under the $\beta$-parameterization and the sum constraint. We firstly pointed out the conditions under which the estimator exists and is unique, and then stated that the estimator is consistent and asymptotically normal. Specifically, we showed that the asymptotic variance of the estimator is the pseudoinverse of the Fisher information matrix. Similar results have been proved by \cite{bradley1954incomplete} and \cite{dykstra1960rank}, who showed that under the $w$-parameterization, the estimated variance of the MLE is the reflexive generalized inverse of the Fisher information matrix;
%\footnote{The term ``reflexive generalized inverse" was not explicitly used by \cite{bradley1954incomplete} and \cite{dykstra1960rank}.}
our result, and its proof, however, is of special interest not only because it considers a different kind of parameterization and a more flexible data collection design, but also because it contains some insights about the asymptotic properties of a constrained MLE for an unidentifiable statistical model.

In particular, we pointed out that under the $\beta$-parameterization, using the sum constraint has the unique advantage of minimizing the sum of the variances of all the estimators. % The natural analog of the sum constraint under the $\beta$-parameterization is a ``product" constraint under the $w$-parameterization. Although the sum constraint under the $w$-parameterization can be found in the literature, it is not necessarily the ``optimal" constraint that minimizes the sum of the variances (derivation not shown).% As for the $w$-parameterization, while the sum constraint does prevent the uncertainty of one estimator from being propagated to all the other estimators, we found (derivation not shown) that it is not necessarily the ``optimal" constraint that minimizes the sum of the variances -- the latter  depends on the true value of $\mathbf{w}$. 
We showed in an empirical example that poor choice of the reference object for the reference constraint can artificially propagate uncertainty to other objects; this behavior is avoided when using the sum constraint. Thus, we strongly recommend using the sum constraint in practice.

\section*{Acknowledgements}

This work has been supported (in part) by Grant \#2003-22461 from the Russell Sage Foundation. Any opinions expressed are those of the authors alone and should not be construed as representing the opinions of the Foundation.

\begin{appendix}
\begin{comment}
\section{The pseudoinverse matrix and its properties\label{appendix-A}}
This section review the definition and properties of the pseudoinverse matrix.  For $\mathbf{A} \in \mathbb{R}^{n \times n}$, a matrix $\mathbf{X} \in \mathbb{R}^{n \times n}$ is called the pseudoinverse of $\mathbf{A}$, denoted as $\mathbf{X = A^\dagger} $, if all of the following conditions hold:
\begin{enumerate}

\item $\mathbf{AXA = A}$;
\item $\mathbf{XAX = X}$;
\item $\mathbf{(AX)^\top = AX}$;
\item $\mathbf{(XA)^\top = XA}$.

\end{enumerate}
If only the first two condition holds, then $\mathbf{X}$ is called a reflexive generalized inverse of $\mathbf{A}$. 
An important property of the pseudoinverse matrix is that $\mathbf{A}^\dagger$ shares the same column space and null space with $\mathbf{A}$. For these terms and properties in matrix theory, we refer to \cite{ben2003generalized}. 
\end{comment}

\section{Proof of Corollary \ref{corollary-1}} \label{appendix-A}

\begin{proof}

Let the MLE for the Bradley-Terry model under the $w$-parameterization and the sum constraint be denoted by
\begin{equation}
\hat{\mathbf{w}}(\mathbf{W}) = \mathop{\text{argmax}}_{\mathbf{1^\top w} = 1} \ell(\mathbf{w;W})
\label{w-mle}
\end{equation}
for the log-likelihood function
$
\ell(\mathbf{w;W}) = \sum_{i \neq j} W_{ij} \log \left(w_i / (w_i+w_j)\right).
$
For any $\mathbf{w}$ such that $\mathbf{1}^\top \mathbf{w} = 1$ and any $\aalpha$ such that $\mathbf{1}^\top \aalpha \neq 0$, define
\[
f_{\aalpha}(\mathbf{w}) = \log \mathbf{w} - \frac{\aalpha^\top \log \mathbf{w}}{\mathbf{1}^\top \aalpha}\mathbf{1}.
\]
where $\log(\mathbf{w})=(\log w_1,\log w_2,...,\log w_n)$ is an element-wise operation. It is easy to verify that
\begin{enumerate}
\item $\aalpha^\top f_{\alpha}(\mathbf{w}) = 0$;
\item If $\bbeta = f_{\aalpha}(\mathbf{w})$, then $\ell(\mathbf{w;W}) = \ell(\bbeta;\mathbf{W})$;
\item For every $\bbeta$ such that $\aalpha^\top \bbeta = 0$, we have $f_{\aalpha}^{-1}(\bbeta) = \frac{\exp(\bbeta)}{\mathbf{1}^\top \exp(\bbeta)}$, with $\exp(\cdot)$ an element-wise operation.% = (\exp(\beta_1),\exp(\beta_2),...,\exp(\beta_n))$.
\end{enumerate}
Consequently, $f_\alpha$ is a bijection from the set $\{\mathbf{w}:\mathbf{1}^\top \mathbf{w} = 1\}$ to the set $\{\bbeta:\aalpha^\top \bbeta = 0\}$  that preserves the value of the log-likelihood function. Therefore, the following two conditions are equivalent:
 \begin{enumerate}
     \item There exists a unique $\mathbf{w}^*$ such that $\mathbf{1}^\top \mathbf{w}^* = 1$ and
 $
     \ell(\mathbf{w}^*;\mathbf{W}) = \max_{\mathbf{1}^\top \mathbf{w}} \ell(\mathbf{w};\mathbf{W}).
  $
     \item There exists a unique $\bbeta^*$ such that $\aalpha^\top \bbeta = 0$ and 
  $
     \ell(\bbeta^*;\mathbf{W}) = \max_{\aalpha^\top \bbeta = 0} \ell(\bbeta;\mathbf{W}).
  $
 \end{enumerate}
Thus the existence and uniqueness of the MLE under the $\beta$-parameterization \eqref{bt-mle} is equivalent to the existence and uniqueness of the MLE for $w$-parameterization under the sum constraint \eqref{w-mle}, for which Theorem \ref{theorem-1} provides the necessary and sufficient condition (Assumption \ref{assumption-A}).
\end{proof}

%\section{Proof of Lemma \ref{lemma-1}\label{appendix-B}}
% In this section, we give a proof of Lemma \ref{lemma-1}.

\section{Proof of Lemma \ref{lemma-2}\label{appendix-C}}
In this section, we give a proof of Lemma \ref{lemma-2}. This proof is based on the properties of the pseudoinverse matrix and the following lemma.

\begin{lemma}\label{lemma-1}
If Assumption \ref{assumption-A} holds, then for any $\bbeta \in \mathbb{R}^n$, the Hessian $H(\bbeta)$ is negative semi-definite, and
\[
\ker(H(\bbeta)) = \textnormal{span}(\mathbf{1}).
\]
\end{lemma}

\emph{Proof of Lemma \ref{lemma-1}.} For any $\mathbf{x} = (x_1,x_2,...,x_n)^\top \in \mathbb{R}^n$, 
\begin{equation}
\begin{split}
 & \mathbf{x}^\top H(\bbeta,\mathbf{W}) \mathbf{x}\\
&= -\sum_{1\leq i < j \leq n} V_{ij} \sigma(\beta_i - \beta_j)\sigma(\beta_j - \beta_i) \mathbf{x}^\top (\mathbf{e}_i - \mathbf{e}_j) (\mathbf{e}_i - \mathbf{e}_j) ^\top \mathbf{x}\\
&= -\sum_{1\leq i < j \leq n} V_{ij} \sigma(\beta_i - \beta_j)\sigma(\beta_j - \beta_i) (x_i - x_j)^2 \leq 0.
\end{split}
\end{equation}
Furthermore, in order for the equality to be reached, for any pair $(i,j)$ such that $V_{ij}>0$, it must be true that $x_i = x_j$. Under Assumption \ref{assumption-A}, the comparison graph $\mathcal{G}$ is strongly connected. This effectively means that all entries of $\mathbf{x}$ are the same. In other words, $\mathbf{x} \in \text{span}(\mathbf{1})$. So in summary, $H(\bbeta,\mathbf{W})$ is negative semi-definite, and $H(\bbeta,\mathbf{W}) \mathbf{x} = 0$ if and only if $\mathbf{x} \in \text{span}(\mathbf{1})$. 
\qed

\emph{Proof of Lemma \ref{lemma-2}}
Lemma \ref{lemma-1} shows that for any $\tilde{\bbeta} \in \mathbb{R}^n$, the null space of $H(\tilde{\bbeta})$ is $\mathop{\textnormal{span}}(\mathbf{1})$. Therefore, if $\mathbf{1}^\top \bbeta = 0$, then $\bbeta$ is in the column space of $H(\tilde{\bbeta})$. The pseudoinverse of a matrix  has the same column space as the original matrix  \citep{ben2003generalized}, and thus $\bbeta$ is in the column space of $H^\dagger(\tilde{\bbeta})$. In other words, there exists $\mathbf{x} \in \mathbb{R}^n$, such that
\begin{equation}
\bbeta = H^\dagger(\tilde{\bbeta})\mathbf{x}.
\end{equation}
Therefore, by the definition of the pseudoinverse matrix,
\begin{equation}
H^\dagger(\tilde{\bbeta})H(\tilde{\bbeta})\bbeta = H^\dagger(\tilde{\bbeta})H(\tilde{\bbeta})H^\dagger(\tilde{\bbeta})\mathbf{x}= H^\dagger(\tilde{\bbeta})\mathbf{x}= \bbeta.
\end{equation}
Furthermore, since $H^\dagger(\tilde{\bbeta})$ has the same null space as $H(\tilde{\bbeta})$, we have that $H^\dagger(\tilde{\bbeta})\mathbf{1} = 0$.
\qed

\end{appendix}

%\section*{References}
\setlength{\bibsep}{0pt plus 0.3ex}

\bibliography{elsarticle-template/mybibfile}

\end{document}